\documentclass[a4paper,10pt]{article}
\usepackage{amsmath,amssymb,amsfonts,amsthm}
\usepackage[usenames,dvipsnames]{color}
\usepackage{a4wide}
\usepackage[verbose,colorlinks=true,linktocpage=true,linkcolor=blue,citecolor=blue]{hyperref}

\theoremstyle{definition}
\newtheorem{definition}{Definition}[section]

\theoremstyle{plain}
\newtheorem{theorem}[definition]{Theorem}
\newtheorem{proposition}[definition]{Proposition}
\newtheorem{lemma}[definition]{Lemma}
\newtheorem{corollary}[definition]{Corollary}
\newtheorem{example}[definition]{Example}
\theoremstyle{remark}
\newtheorem{remark}[definition]{Remark}

\numberwithin{equation}{section}

\allowdisplaybreaks[4]

\begin{document}

\title{Remarks on the Second Homology Groups of Queer Lie Superalgebras}
\author{Yongjie Wang ${}^1$ and Zhihua Chang${}^2$\footnote{Corresponding Author: Zhihua Chang, Email: mazhhchang@scut.edu.cn}}
\maketitle

\begin{center}
\footnotesize
\begin{itemize}
\item[1] School of Mathematics, Hefei University of Technology, Hefei, Anhui, 23009, China.
\item[2] School of Mathematics, South China University of Technology, Guangzhou, Guangdong, 510640, China.
\end{itemize}
\end{center}

\begin{abstract}
The aim of this note is to completely determine the second homology group of the special queer Lie superalgebra $\mathfrak{sq}_n(R)$ coordinatized by a unital associative superalgebra $R$, which will be achieved via an isomorphism between the special linear Lie superalgebra $\mathfrak{sl}_{n}(R\otimes Q_1)$ and the special queer Lie superalgebra $\mathfrak{sq}_n(R)$. 
\bigskip

\bigskip

\end{abstract}

\section{Introduction}
\label{sec:intr}
The universal central extension of a Lie (super)algebra plays an important role in Lie theory. On one hand, an affine Kac-Moody Lie algebra over the field $\mathbb{C}$ can be realized as the universal central of a twisted loop algebra based on a finite-dimensional Lie algebra (extended by a derivation). The realizations of extended affine Lie algebras also involve the universal central extension of an infinite-dimensional Lie algebra such as a twisted multi-loop algebra or a matrix Lie algebra coordinatized by quantum tori. On the other hand, an explicit description of the kernel of the universal central extension provides us with the second homology group of this Lie (super)algebra with coefficients in the trivial module \cite{Neher2003}. From this point of view, C. Kassel and J. L. Loday studied in \cite{KasselLoday1982} the universal central extension of $\mathfrak{sl}_n(A)$ with $A$ a unital associative algebra and established an elegant isomorphism between the second homology group of $\mathfrak{sl}_n(A)$ for $n\geqslant5$ and the first cyclic homology group $\mathrm{HC}_1(A)$. This establishes a connection between the Lie theoretical object $\mathfrak{sl}_n(A)$ and the K-theoretical object $\mathrm{HC}_1(A)$ in non-commutative geometry. This phenomenon also appears in unitary Lie algebras \cite{Gao1996}, where an isomorphism between the second homology group of a unitary Lie algebra and the first skew-dihedral homology group of the corresponding associative algebra was established.

Inspired by these works, the universal central extension of Lie superalgebras have been studied extensively in recent decades. For a Lie superalgebra $\mathfrak{g}$ of classical type over a field of characteristic zero, K. Iohara and Y. Koga determined in \cite{IoharaKoga2001,IoharaKoga2005} the second homology group of $\mathfrak{g}\otimes R$ with $R$ a super-commutative unital associative superalgebra. In the case where $\mathfrak{g}$ is a Lie superalgebra with a non-degenerate invariant bilinear form and $R$ a super-commutative unital associative superalgebra, K. H. Neeb and M. Yousofzadeh in \cite{NY2020} also obtained the universal central extension of $\mathfrak{g}\otimes R$ by explicitly creating the corresponding $2$-cocycle with K\"{a}hler differentials. 

Matrix Lie superalgebras that are coordinatized by associative superalgebras have also attracted algebraists' attentions. The universal central extensions of $\mathfrak{sl}_n(R), n\geqslant3$ have been completely determined in \cite{ChenGuay2013}, which has been further generalized to the Lie superalgebra $\mathfrak{sl}_{m|n}(R)$, $m+n\geqslant3$ in \cite{ChenSun2015} and \cite{GL2017}. The authors addressed in \cite{ChangWang2016} and \cite{ChangChengWang2018} the case of an otho-symplectic Lie superalgebra and the case of a periplectic Lie superalgebra that are coordinatized by an associative superalgebra with a super anti-involution. These Lie superalgebras are Super analogy of unitary Lie algebra, thus it can help us to identify the second homology groups of these Lie superalgebras with the $\mathbb{Z}/{2\mathbb{Z}}$-graded version of the first dihedral and skew-dihedral homology group.

This short note aims to determine the second homology group of a special queer Lie superalgebra. It has been shown in \cite{MartinezZelmanov2003} that every $Q(n-1)$-graded  Lie superalgebra is centrally isogenous to the Steinberg Lie superalgebra $\mathbf{st}_{n}(S)$ for $n\geqslant4$, where $S$ is a unital associative superalgebra such that its odd part contains an element $\nu$ with $\nu^2=1$. We will show in Section~\ref{sec:sq} that such a Lie superalgebra can be characterized by the special queer Lie superalgebra $\mathfrak{sq}_n(R)$ coordinatized by an associative algebra $R$. Moreover, we observe that the special queer Lie superalgebra $\mathfrak{sq}_n(R)$ is indeed isomorphic to the special linear Lie superalgebra $\mathfrak{sl}_{n}(R\otimes Q_1)$, where $Q_1$ is the Clifford superalgebra of rank one (see \eqref{eq:Q1} below).

If $R$ is a super-commuative superalgebra over $\Bbbk$, where $\Bbbk$ is a commutative base ring with $2$ invertible, then $\mathfrak{sq}_n(R)$ is isomorphic to $\mathfrak{sq}_n(\Bbbk)\otimes R$. Their universal central extensions have been determined in \cite{IoharaKoga2005}, \cite{MikhalevPinchuk2005} and \cite{NY2020} by using different methods. We deal with the  general case where $R$ is an arbitrary unital associative superalgebra.  The isomorphism between $\mathfrak{sq}_n(R)$ and $\mathfrak{sl}_n(R\otimes Q_1)$ allows us to identify the second homology group of $\mathfrak{sq}_n(R)$ with the first $\mathbb{Z}/2\mathbb{Z}$-graded cyclic homology group $\mathrm{HC}_1(R\otimes Q_1)$. By further identifying $\mathrm{HC}_1(R\otimes Q_1)$ with $\mathrm{HC}_1(R)$ (up to a parity change), we obtain the second homology group of $\mathfrak{sq}_n(R)$. We remark that one could also introduce the Steinberg queer Lie superalgebra and follow the same lines as C. Kassel and J. L. Loday's argument to obtain this result, but it would be tedious. 

The remainder of this note is divided into three sections. In Section~\ref{sec:sq}, we will define the queer Lie superaglebra $\mathfrak{q}_n(R)$, the special queer Lie superalgebra $\mathfrak{sq}_n(R)$, and prove the isomorphism between $\mathfrak{sq}_n(R)$ and $\mathfrak{sl}_n(R\otimes Q_1)$ for an arbitrary unital associative superalgebra $R$. Section~\ref{sec:HC} is devoted to identifying $\mathrm{HC}_1(R\otimes Q_1)$ with $\mathrm{HC}_1(R)$ up to a parity change. The second homology group of $\mathfrak{sq}_n(R)$ will be discussed in Section~\ref{sec:H2}. 

Throughout this note, we always assume that $\Bbbk$ is a unital commutative associative base ring with $2$ invertible. All modules, associative (super)algebras and Lie (super)algebras are over $\Bbbk$. We write $\mathbb{Z}/2\mathbb{Z}=\{\bar{0},\bar{1}\}$ and the parity of an element $x$ is denoted by $|x|$. If both $A$ and $B$ are associative superalgebras, then $A\otimes B$ is understood as the associative superalgebra with graded multiplication
$$(a_1\otimes b_1)(a_2\otimes b_2)=(-1)^{|a_2||b_1|}a_1a_2\otimes b_1b_2, \text{ for homogeneous elements } a_2\in A, b_1\in B.$$

 \section{Queer Lie superalgebras}
 \label{sec:sq}

It is shown in \cite{MartinezZelmanov2003} that every $Q(n-1)$-graded Lie superalgebra is centrally isogenous to the Steinberg Lie superalgebra $\mathfrak{st}_{n+1}(S)$, where $S$ is a unital associative superalgebra such that there exists $\nu\in S_{\bar{1}}$ satisfying $\nu^2=1$. In fact, such an associative superalgebra is isomorphic to $R\otimes Q_1$, where $R$ is a (non-super) associative algebra and $Q_1$ is the associative superalgebra
\begin{equation}
Q_1=\left\{\begin{pmatrix}a&b\\b&a\end{pmatrix}\middle|a,b\in\Bbbk\right\}.\label{eq:Q1}
\end{equation}
The associative superalgebra $Q_1$ has one dimensional even part spanned by $1=\begin{pmatrix}1&0\\0&1\end{pmatrix}$ and one dimensional odd part spanned by $\nu=\begin{pmatrix}0&1\\1&0\end{pmatrix}$. 

Let $R$ be a unital associative superalgebra over $\Bbbk$ and $\mathrm{M}_{m|n}(R)$ be the associative superalgebra of all $(m+n)\times(m+n)$-matrices with entries in $R$, in which the parity of a matrix unit $e_{ij}(a)$ (with $a\in R$ at the $(i,j)$-position and $0$ elsewhere) is 
\begin{equation*}
	|e_{ij}(a)|:=|i|+|j|+|a|,\quad a\in R,\quad 1\leqslant i,j\leqslant m+n,
	\label{eq:parity_matrix}
\end{equation*}
and
\begin{equation*}
	|i|=\begin{cases}0,&\text{if }i\leqslant m,\\1,&\text{if }i>m.\end{cases}
	\label{eq:parity_index}
\end{equation*}
The associative superalgebra $\mathrm{M}_{m|n}(R)$ is naturally regraded as a Lie superalgebra under the standard super-commutator. We denote this Lie superalgebra by $\mathfrak{gl}_{m|n}(R)$. Define \textit{the queer Lie superalgebra coordinatized by $R$} to be the Lie sub-superalgebra
	$$\mathfrak{q}_n(R)=\left\{\begin{pmatrix}A&B\\\rho(B)&\rho(A)\end{pmatrix}\middle| A,\ B\in \mathrm{M}_n(R)\right\}\subseteq\mathfrak{gl}_{n|n}(R),$$
	where $\rho(A)=(\rho(a_{ij}))_{n\times n}$ if $A=(a_{ij})_{n\times n}$ and $\rho(a)=(-1)^{|a|}\rho(a)$ for homogeneous element $a\in R$.
The queer Lie superalgebra $\mathfrak{q}_n(R)$ is not perfect even when $R=\Bbbk$ is the base ring. Instead, we call its derived sub-superalgebra \textit{the special queer Lie superalgebra}: 
\begin{equation}
\mathfrak{sq}_n(R)=:[\mathfrak{q}_n(R),\mathfrak{q}_n(R)].
\end{equation}

If $R$ is a non-super unital associative algebra, we will see that $\mathfrak{sq}_n(R)$ is indeed a $Q(n-1)$-graded Lie superalgebra by showing that it is centrally isogenous to $\mathfrak{st}_n(S)$ for some $S$. On the other hand, it is known from \cite{ChenGuay2013} that the Steinberg Lie superalgebra $\mathfrak{st}_n(R\otimes Q_1)$ is a central extension of $\mathfrak{sl}_n(R\otimes Q_1)$ for $n\geqslant3$.  In fact, we can prove the following proposition which connects $\mathfrak{sq}_n(R)$ and $\mathfrak{sl}_n(R\otimes Q_1)$.

\begin{proposition}
\label{ex: gltoq}
Let $R$ be an arbitrary unital associative superalgebra. Then  there is an isomorphism of Lie superalgebras
\begin{equation*}
\mathfrak{q}_n(R)\cong\mathfrak{gl}_{n}(R\otimes Q_1).
\end{equation*}
Consequently,
 $\mathfrak{sq}_n(R)\cong\mathfrak{sl}_{n}(R\otimes Q_1)$ as Lie superalgebras.
\end{proposition}
\begin{proof}
We use $e_{ij}(a)$ to denote the $n\times n$-matrix with $a\in R$ at the $(i,j)$-position and $0$ elsewhere. Then $\mathfrak{q}_n(R)$ is spanned by
$$u_{ij}(a)=\begin{pmatrix}e_{ij}(a)&0\\0&(-1)^{|a|}e_{ij}(a)\end{pmatrix},\qquad 
w_{ij}(a)=\begin{pmatrix}0&e_{ij}(a)\\ (-1)^{|a|}e_{ij}(a)&0\end{pmatrix},$$
for $i,j=1,\ldots,n$ and $a\in R$ is homogeneous . They satisfy
\begin{align*}
&[u_{ij}(a),u_{kl}(b)]
=\delta_{jk}u_{il}(ab)-(-1)^{|a||b|}\delta_{il}u_{kj}(ba),\\
&[u_{ij}(a),w_{kl}(b)]
=\delta_{jk}w_{il}(ab)-(-1)^{|a||b|}\delta_{il}w_{kj}(ba),\\
&[w_{ij}(a),w_{kl}(b)]=(-1)^{|b|}\left(\delta_{jk}u_{il}(ab)+(-1)^{|a||b|}\delta_{il}u_{kj}(ba)\right).
\end{align*}
Hence, the $\Bbbk$-linear map
\begin{align*}
\mathfrak{q}_n(R)\rightarrow\mathfrak{gl}_n(R\otimes Q_1),\quad
u_{ij}(a)\mapsto e_{ij}(a\otimes 1),\quad w_{ij}(a)&\mapsto e_{ij}(a\otimes\nu),
\end{align*}
gives the desired isomorphism of Lie superalgebras.
\end{proof}

\begin{corollary}
\label{cor:perfect}
The Lie superalgebra $\mathfrak{sq}_n(R)$ is perfect for $n\geqslant2$ and it can be described as
\begin{equation}
	\mathfrak{sq}_n(R)=\left\{\begin{pmatrix}A&B\\ \rho(B)&\rho(A)\end{pmatrix}|\ \mathrm{Tr}(B)\in [R, R]\right\}.\label{eq:sq}
\end{equation}
\end{corollary}
\begin{proof}
Since $\mathfrak{sq}_n(R)\cong\mathfrak{sl}_n(R\otimes Q_1)$, the perfectness of $\mathfrak{sq}_n(R)$ follows from the perfectness of $\mathfrak{sl}_n(R\otimes Q_1)$. 

Now, the Lie superalgebra $\mathfrak{sl}_n(R\otimes Q_1)$ can be characterized as
$$\mathfrak{sl}_n(R\otimes Q_1)=\left\{X\in\mathfrak{gl}_n(R\otimes Q_1)\middle|\mathrm{Tr}(X)\in [R\otimes Q_1, R\otimes Q_1]\right\}$$
Note that
$$[1\otimes\nu,1\otimes\nu]=2(1\otimes1),\text{ and }[a\otimes1,b\otimes\nu]=[a,b]\otimes\nu.$$
Since $2$ is invertible, we deduce that $[R\otimes Q_1, R\otimes Q_1]=(R\otimes1)\oplus([R,R]\otimes\nu)$. Hence, for an element $X=\sum_{i,j}e_{ij}(a_{ij}\otimes1+b_{ij}\otimes\nu)\in\mathfrak{gl}_n(R\otimes Q_1)$, $$\mathrm{Tr}(X)=\sum_i (a_{ii}\otimes1+b_{ii}\otimes\nu)\in[R\otimes Q_1,R\otimes Q_1]$$
if and only if $\sum_ib_{ii}\in[R,R]$. Hence, the preimage of $X$ in $\mathfrak{sq}_n(R)$ is the matrix $\begin{pmatrix}A&B\\\rho(B)&\rho(A)\end{pmatrix}$ where $A=(a_{ij})$, $B=(b_{ij})$ such that $\mathrm{Tr}(B)\in[R,R]$.
\end{proof}

\begin{remark}
\label{rmk:p1_impft}
\begin{enumerate}
\item The Lie superalgebra $\mathfrak{sq}_1(R)$ is not necessarily perfect. For instance, if $R$ is  super-commutative, then
	$$\mathfrak{q}_1(R):=\left\{\begin{pmatrix} a&b\\\rho(b)&\rho(a)\end{pmatrix}\middle| a,b\in R\right\},\text{ and }\mathfrak{sq}_1(R)=\left\{\begin{pmatrix}a&0\\0&\rho(a)\end{pmatrix}\middle| a\in R\right\}.$$
The Lie superalgebra $\mathfrak{sq}_1(R)$ is not perfect since $[\mathfrak{sq}_1(R),\mathfrak{sq}_1(R)]=0$.
\item In \cite{Gao1993}, the author determined the second homology group of Lie algebra $\mathfrak{sl}_2(S)$ whether $S$ is an associative algebra or not. More concretely, the homology group $H_2(\mathfrak{sl}_2(S),\Bbbk)$ is the generalization of cyclic homology group $\mathrm{HC}_1(S)$ denoted by $\mathrm{hC}_1(S)$ (more details could be found in \cite[Theorem 2]{Gao1993}) when $S$ is an associtive algebra. For nonassociative algebra $S$, the vector space $\mathfrak{sl}_2(S)$ is a Lie algebra if and only if $S$ satisfies certain conditions, that is $\mathfrak{sl}_2$-admissible.  	In this case, the homology group was also computed (see \cite[Proposition 4]{Gao1993}). Analogue to $\mathfrak{sl}_2$ case, the definition of $\mathfrak{sq}_2(R)$ for a (non)associative superalgebra $R$ and its second homology group should be considered separately.  
\item As C. Martinez and E. I. Zelmanov pointed out in \cite{MartinezZelmanov2003}, a $Q(2)$-graded  Lie superalgebra is centrally isogenous to $\mathbf{st}_{3}(R)$, where $R$ is a unital alternative superalgebra that is not necessarily associative. But one can check that the isomorphism in Proposition \ref{ex: gltoq} is still valid if $R$ is a unital alternative superalgebra. 
\end{enumerate}
\end{remark}
When $R$ is specified to be a concrete unital associative superalgebra, we can find many well acquainted examples in these queer Lie superalgebras. 

\begin{example}
	\label{ex:comm}
Let $R$ be a unital super-commutative associative superalgebra, then 
\begin{equation}
	\mathfrak{q}_n(R)\cong\mathfrak{q}_n(\Bbbk)\otimes_{\Bbbk}R,
\end{equation}
where the bracket of $\mathfrak{q}_n(\Bbbk)\otimes_{\Bbbk}R$ is given by
\begin{equation*}
	[x\otimes a,y\otimes b]=(-1)^{|a||y|}[x,y]\otimes ab,\quad x,y\in\mathfrak{q}_n(\Bbbk),\quad a,b\in R.
\end{equation*}
In particular, this gives the untwisted loop queer Lie superalgebra when $R$ is the algebra of Laurent polynomials.
\end{example}

\begin{example}
	\label{ex:qtogl}
	Suppose that $\Bbbk$ contains $\sqrt{-1}$. Let $R$ be a unital associative superalgebra, then  there is an isomorphism of Lie superalgebras 
	\begin{equation*}
		\mathfrak{q}_n(R\otimes Q_1)\cong\mathfrak{gl}_{n|n}(R),
	\end{equation*}
given by 
\begin{align*}
	u_{ij}(r\otimes 1)&\mapsto e_{ij}(r)+(-1)^{|r|}e_{n+i, n+j}(r),&u_{ij}(r\otimes \nu)&\mapsto e_{i,n+j}(r)+(-1)^{|r|}e_{n+i, j}(r),\\
	w_{ij}(r\otimes 1)&\mapsto \sqrt{-1}(-e_{i,n+j}(r)+(-1)^{|r|}e_{n+i, j}(r)),&w_{ij}(r\otimes \nu)&\mapsto \sqrt{-1}(e_{i,j}(r)-(-1)^{|r|}e_{n+i, n+j}(r)),
\end{align*}
where $u_{ij}, w_{ij}$ have the same meaning as in Corollary~\ref{cor:perfect}. Consequently, $\mathfrak{sq}_n(R\otimes Q_1)\cong\mathfrak{sl}_{n|n}(R)$ as Lie superalgebras.
\end{example}

\section{The first $\mathbb{Z}/2\mathbb{Z}$-graded cyclic homology}
\label{sec:HC}

We have seen from Section~\ref{sec:sq} that the queer Lie superalgebra $\mathfrak{sq}_n(R)$ is perfect for $n\geqslant3$ and is isomorphic to $\mathfrak{sl}_n(R\otimes Q_1)$. It has been demonstrated in  \cite{ChenGuay2013} that, if $n\geqslant5$, the universal central extension of $\mathfrak{sl}_n(R\otimes Q_1)$ can be described by the so-called Steinberg Lie superalgebra $\mathfrak{st}_n(R\otimes Q_1)$, and the kernel of  the canonical homomorphism 
$$\pi: \mathfrak{st}_n(R\otimes Q_1)\rightarrow\mathfrak{sl}_n(R\otimes Q_1)$$
is isomorphic to the first $\mathbb{Z}/2\mathbb{Z}$-graded cyclic homology group $\mathrm{HC}_1(R\otimes Q_1)$. We further identify it with the first cyclic homology group $\mathrm{HC}_1(R)$ up to a parity change.

We first recall the definition of $\mathrm{HC}_1(R)$ (we refer to \cite{Kassel1986} for more details about the $\mathbb{Z}/{2\mathbb{Z}}$-graded cyclic homology). Let $I_R$ be the $\Bbbk$-submodule of $R\otimes R$ generated by
$$a\otimes b+(-1)^{|a||b|}b\otimes a,\text{ and }(-1)^{|a||c|}ab\otimes c+(-1)^{|b||a|}bc\otimes a+(-1)^{|c||b|}ca\otimes b$$
for homogeneous $a,b,c\in R$ and $\langle R, R\rangle:=(R\otimes R)/I_R$. We denote by $\boldsymbol{\lambda}(a,b)$ the canonical image of $a\otimes b$ in $\langle R, R\rangle$. Then
$$\mathrm{HC}_1(R):=\left\{\sum_i\boldsymbol{\lambda}(a_i,b_i)\in \langle R,R\rangle\middle| \sum_i[a_i,b_i]=0\right\}.$$
In order to distinguish $\mathrm{HC}_1(R)$ and $\mathrm{HC}_1(R\otimes Q_1)$, we denote by $\mathbf{h}(x,y)$ the canonical image of $x\otimes y$ in $\langle R\otimes Q_1, R\otimes Q_1\rangle$ for $x,y\in R\otimes Q_1$.

\begin{lemma}\label{hrelations}
The following relations hold in $\langle R\otimes Q_1,R\otimes Q_1\rangle$:
\begin{align}
&\mathbf{h}(a\otimes 1, b\otimes \nu)=-(-1)^{|a||b|}\mathbf{h}(b\otimes 1, a\otimes\nu),\quad\text{ for } a,b\in R,\label{eq:hodd}\\
&\mathbf{h}(a\otimes 1, b\otimes 1)=\mathbf{h}(a\otimes \nu, b\otimes \nu)=0,\quad \text{ for } a\in R_{\bar{1}} \text{ or } b\in R_{\bar{1}}, \label{eq:heven1}\\
&\mathbf{h}(a\otimes 1, b\otimes 1)=\frac{1}{2}\mathbf{h}([a,b]\otimes\nu, 1\otimes\nu),\quad \text{ for } a, b\in R_{\bar{0}},\label{eq:heven2}\\
&\mathbf{h}(a\otimes\nu,b\otimes\nu)=\frac{1}{2}\mathbf{h}(\{a,b\}\otimes\nu,1\otimes\nu),\quad \text{ for } a, b\in R_{\bar{0}},\label{eq:heven3}
\end{align}
where $\{, \}$ is possion bracket, i.e.,  $\{a,b\}=ab+ba.$
\end{lemma}
\begin{proof}
It follows from the definition of $\langle R\otimes Q_1,R\otimes Q_1\rangle$ that $\mathbf{h}(x,y)$ satisfies
\begin{gather}
\mathbf{h}(x,y)+(-1)^{|x||y|}\mathbf{h}(y,x)=0,\label{eq:h1}\\
(-1)^{|x||z|}\mathbf{h}(xy,z)+(-1)^{|y||x|}\mathbf{h}(yz,x)+(-1)^{|z||y|}\mathbf{h}(zx,y)=0.\label{eq:h2}
\end{gather}

By setting $y=z=1\otimes1$ in \eqref{eq:h2}, we first observe that 
$$\mathbf{h}(x,1\otimes1)=\mathbf{h}(1\otimes1,x)=0$$
for all $x\in R\otimes Q_1$.

In order to show \eqref{eq:hodd}, we set $x=a\otimes\nu, y=b\otimes\nu, z=1\otimes\nu$ in \eqref{eq:h2}. Then
\begin{equation*}
\mathbf{h}(ab\otimes1,1\otimes \nu)
+(-1)^{|a||b|}\mathbf{h}(b\otimes1,a\otimes \nu)
+\mathbf{h}(a\otimes1,b\otimes \nu)
=0.
\end{equation*}
from which we observe that $\mathbf{h}(a\otimes1,1\otimes\nu)=0$ for all $a\in R$. Hence, \eqref{eq:hodd} follows.

Similarly, set $x=a\otimes\nu, y=b\otimes\nu, z=c\otimes1$ in \eqref{eq:h2}, we obtain
\begin{equation}
(-1)^{|a||c|+|a|}\mathbf{h}(ab\otimes1,c\otimes1)-(-1)^{|b||a|}\mathbf{h}(bc\otimes\nu, a\otimes\nu)+(-1)^{|b||c|+|b|}\mathbf{h}(ca\otimes\nu,b\otimes\nu)=0.\label{eq:h9}
\end{equation}
Let $c=1$ in \eqref{eq:h9}, we have
$$0-(-1)^{|a||b|}\mathbf{h}(b\otimes\nu, a\otimes\nu)+(-1)^{|b|}\mathbf{h}(a\otimes\nu,b\otimes\nu)=0.$$
Then it follows from \eqref{eq:h1} that
\begin{equation}
\left(1-(-1)^{|a|}\right)\mathbf{h}(a\otimes\nu, b\otimes \nu)=0=(1-(-1)^{|a|})\mathbf{h}(b\otimes\nu,a\otimes\nu),\label{eq:h5}
\end{equation}
which ensures that $\mathbf{h}(a\otimes\nu,b\otimes\nu)=0$ if $a$ or $b$ is odd.

Considering the case of $a=1$ and the case of $b=1$ in \eqref{eq:h9} respectively, we deduce that
\begin{align*}
\mathbf{h}(b\otimes1,c\otimes1)-\mathbf{h}(bc\otimes\nu, 1\otimes\nu)+(-1)^{|b||c|+|b|}\mathbf{h}(c\otimes\nu,b\otimes\nu)=0,\\
(-1)^{|a||c|+|a|}\mathbf{h}(a\otimes1,c\otimes 1)-\mathbf{h}(c\otimes\nu, a\otimes \nu)+\mathbf{h}(ca\otimes\nu,1\otimes\nu)=0.
\end{align*}
Equivalently, for all $a,b\in R$,
\begin{gather}
(-1)^{|b|}\mathbf{h}(a\otimes1,b\otimes1)+\mathbf{h}(a\otimes\nu,b\otimes\nu)-\mathbf{h}(ab\otimes\nu,1\otimes\nu)=0,\label{eq:h7}\\
\mathbf{h}(a\otimes1,b\otimes1)+(-1)^{|b|}\mathbf{h}(a\otimes\nu,b\otimes\nu)-\mathbf{h}(ab\otimes\nu,1\otimes\nu)=0.
\end{gather}
These imply that
$$\left(1-(-1)^{|b|}\right)\mathbf{h}(a\otimes1,b\otimes1)=-\left(1-(-1)^{|b|}\right)\mathbf{h}(a\otimes\nu,b\otimes\nu)=0,$$
and hence $\mathbf{h}(a\otimes1,b\otimes1)=0$ if $a$ or $b$ is odd.

For \eqref{eq:heven2} and \eqref{eq:heven3}, we assume that $a,b$ are both even and exchanging $a$ and $b$ in \eqref{eq:h7}, then
\begin{equation}\mathbf{h}(b\otimes1,a\otimes1)+\mathbf{h}(b\otimes\nu, a\otimes\nu)-\mathbf{h}(ba\otimes\nu,1\otimes\nu)=0.\label{eq:h8}\end{equation}
Then \eqref{eq:h1}, \eqref{eq:h7} and \eqref{eq:h8} yield \eqref{eq:heven2} and \eqref{eq:heven3}.\end{proof}

\begin{theorem}
\label{maintheorem}
Let $R$ be an arbitrary unital associative superalgebra over $\Bbbk$. Then 
$$\mathrm{HC}_1(R\otimes Q_1)\cong\mathrm{HC}_1(R)\otimes\Bbbk^{0|1}.$$
\end{theorem}
\begin{proof}
The $\mathbb{Z}/2\mathbb{Z}$-graded vector space $\Bbbk^{0|1}$ has a zero even part and a one-dimensional odd part. Hence, $\mathrm{HC}_1(R)\otimes\Bbbk^{0|1}$ is the $\mathbb{Z}/2\mathbb{Z}$-graded space obtained by exchanging the even and odd parts of $\mathrm{HC}_1(R)$. It suffices to show that there is an odd isomorphism between $\mathrm{HC}_1(R\otimes Q_1)$ and $\mathrm{HC}_1(R)$.

Note that every element in $R\otimes Q_1$ can be written as $a\otimes1+b\otimes\nu$ for $a,b\in R$. By Lemma~\ref{hrelations}, every element in $\langle R\otimes Q_1, R\otimes Q_1\rangle$ can be written as
$$z=\sum_j\mathbf{h}(c_j\otimes\nu,1\otimes\nu)+\sum_i\mathbf{h}(a_i\otimes1,b_i\otimes\nu)$$
where $c_j\in R_{\bar{0}}$, $a_i, b_i\in R$, and both summations run over some finite sets. Such an element $z$ lies in $\mathrm{HC}_1(R\otimes Q_1)$ if and only if
$$0=\sum_j[c_j\otimes\nu,1\otimes\nu]+\sum_i[a_i\otimes1,b_i\otimes\nu]=2\sum_j c_j\otimes1+\sum_i[a_i,b_i]\otimes\nu,$$
which is equivalent to $\sum_j c_j=0$ and $\sum_i[a_i,b_i]=0$. Hence,
$$\mathrm{HC}_1(R\otimes Q_1)=\left\{\sum_j\mathbf{h}(a_j\otimes1,b_j\otimes\nu)\middle|a_j,b_j\in R\text{ such that }\sum_j[a_j,b_j]=0\right\}.$$

Now, we can define two odd $\Bbbk$-linear maps:
\begin{align*}
&\varphi:\mathrm{HC}_1(R\otimes Q_1)\rightarrow\mathrm{HC}_1(R), &&\sum_i\mathbf{h}(a_i\otimes1,b_i\otimes\nu)\mapsto\sum_i\boldsymbol{\lambda}(a_i,b_i),\\
&\psi:\mathrm{HC}_1(R)\rightarrow\mathrm{HC}_1(R\otimes Q_1), &&\sum_i\boldsymbol{\lambda}(a_i,b_i)\mapsto\sum_i\mathbf{h}(a_i\otimes1,b_i\otimes\nu).
\end{align*}
Lemma~\ref{hrelations} imply that both $\varphi$ and $\psi$ are well-defined and they are inverse to each other. Hence, $\varphi$ is an odd isomorphism. In other word, $\mathrm{HC}_1(R\otimes Q_1)\cong\mathrm{HC}_1(R)\otimes\Bbbk^{0|1}$ as $\mathbb{Z}/2\mathbb{Z}$-graded $\Bbbk$-modules.
\end{proof}

\begin{remark}
If $R=\Bbbk$, $R\otimes Q_1\cong Q_1$ is the Clifford superalgebra associated to the quadratic form $q(x)=x^2$. Its $\mathbb{Z}/2\mathbb{Z}$-graded cyclic homology $\mathrm{HC}_*(Q_1)$ has been fully understood in \cite{Kassel1986}. For a unital associative superalgebra $R$, the higher degree $\mathbb{Z}/2\mathbb{Z}$-graded cyclic homology $\mathrm{HC}_n(R\otimes Q_1)$ with $n\geqslant2$ can also be computed through K\"{u}nneth formula established in \cite{Kassel1986}.
\end{remark}

\section{Second homology of queer Lie superalgebras}
\label{sec:H2}

The second homology group of a Lie superalgebra is identified with the kernel of its universal central extension. Based on the interpretation of $\mathrm{HC}_1(R\otimes Q_1)$ in Section~\ref{sec:HC}, we obtain the second homology group of the queer Lie superalgebra $\mathfrak{sq}_n(R)$.

\begin{theorem}
Let $R$ be a unital associative superalgebra and $n\geqslant3$. Then the second homology of the queer Lie superalgebra $\mathfrak{sq}_n(R)$ is 
$$\mathrm{H}_2(\mathfrak{sq}_n(R))\cong\mathrm{HC}_1(R)\otimes\Bbbk^{0|1}.$$
\end{theorem}
\begin{proof}
Note that $\mathfrak{sq}_n(R)\cong\mathfrak{sl}_n(R\otimes Q_1)$, the proof can be done by applying the results obtained in \cite{ChenGuay2013}. In fact, it is shown in \cite{ChenGuay2013} that
\begin{equation}
\mathrm{H}_2(\mathfrak{sl}_n(R\otimes Q_1))=\begin{cases}\mathrm{HC}_1(R\otimes Q_1),&n\geqslant 5,\\
\mathrm{HC}_1(R\otimes Q_1)\oplus (\mathcal{A}_2)^{\oplus 6}, &n=4,\\
\mathrm{HC}_1(R\otimes Q_1)\oplus (\mathcal{A}_3)^{\oplus 6}, &n=3,
\end{cases}
\end{equation}
where $\mathcal{A}_n=(R\otimes Q_1)/\mathcal{I}_n$ and $\mathcal{I}_3$ is the two-sided ideal of  $R\otimes Q_1$ generated by $n(R\otimes Q_1)$ and $[R\otimes Q_1,R\otimes Q_1]$. 

Now, $2$ is invertible in $\Bbbk$, so $[1\otimes\nu,1\otimes\nu]=2(1\otimes1)\in[R\otimes Q_1, R\otimes Q_1]$ is a unit in $R\otimes Q_1$. i.e., $\mathcal{A}_2=\mathcal{A}_3=0$. By Theorem~\ref{maintheorem},
$$\mathrm{H}_2(\mathfrak{sq}_n(R))\cong\mathrm{H}_2(\mathfrak{sl}_n(R\otimes Q_1))\cong\mathrm{HC}_1(R\otimes Q_1)\cong\mathrm{HC}_1(R)\otimes\Bbbk^{0|1}.$$
\end{proof}

\begin{remark}
The Steinberg Lie superalgebra $\mathfrak{st}_3(R)$ and $\mathfrak{st}_4(R)$ might fail to be centrally closed for some unital associative superalgebra $R$. Nonetheless, $\mathfrak{st}_3(R\otimes Q_1)$ and $\mathfrak{st}_4(R\otimes Q_1)$ are always centrally closed for arbitrary unital associative superalgebra $R$.
\end{remark}

Now, we apply the above results to a couple of examples.

In the case where $R$ is super commutative. $\mathfrak{sq}_n(R)\cong\mathfrak{sq}_n(\Bbbk)\otimes R$ and $\mathrm{HC}_1(R)=\Omega^1(R)/\mathrm{d}\Omega^1(R)$ is the module of K\"{a}hler differentials modulo exact ones. Hence,
$$\mathrm{H}_2(\mathfrak{sq}_n(\Bbbk)\otimes R)=\Omega^1(R)/\mathrm{d}(R)\otimes\Bbbk^{0|1},\text{ for }n\geqslant3.$$
which has been proved in \cite{NY2020} by explicitly creating the 2-cocycle and the corresponding universal central extension. It was also been proved in \cite{MikhalevPinchuk2005} by introducing the so-called Steinberg queer Lie superalgebra. 

Note that $\mathfrak{sq}_n(\Bbbk)$ for $n\geqslant 3$ is not a simple Lie superalgebra, it contains a non-trivial center consisting of scalar matrices $\Bbbk I_{n|n}$, where $I_{n|n}$ is the identity matrix.  The quotient $\mathfrak{psq}_n(\Bbbk)=\mathfrak{sq}_n(\Bbbk)/\Bbbk I_{n|n}$ is then a simple Lie superalgebra if $\Bbbk$ is an algebraic closed field of characteristic zero and $n\geqslant 3$. Similarly, if $R$ is super-commutative, the scalar matrices with entries in $R$ is central in $\mathfrak{sq}_n(R)$ and $\mathfrak{psq}_n(\Bbbk)\otimes R=\mathfrak{sq}_n(R)/(RI_{n|n})$. For $n\geqslant3$,
$$\mathrm{H}_2(\mathfrak{psq}_n(\Bbbk)\otimes R)=R\oplus(\mathrm{HC}_1(R)\otimes\Bbbk^{0|1}),$$
which is the result given in \cite{IoharaKoga2005}.

Finally, we suppose that $\Bbbk$ contains $\sqrt{-1}$ and consider $R=S\otimes Q_1$, where $S$ is an arbitrary unital associative superalgebra. We have already known from Example~\ref{ex:qtogl} that $\mathfrak{sq}_n(R)\cong\mathfrak{sl}_{n|n}(S)$. Hence, we conclude that
$$\mathrm{H}_2(\mathfrak{sl}_{n|n}(S))\cong\mathrm{H}_2(\mathfrak{sq}_n(S\otimes Q_1))\cong\mathrm{HC}_1(S\otimes Q_1)\otimes\Bbbk^{0|1}\cong\mathrm{HC}_1(S)\otimes\Bbbk^{0|1}\otimes\Bbbk^{0|1}\cong\mathrm{HC}_1(S),$$
for $n\geqslant3$. This is part of the result obtained in \cite{ChenSun2015} and \cite{GL2017}. 

\section*{Acknowledgments}
Zhihua Chang was supported by National Natural Science Foundation of China (No. 11771455 and 12071150) and Guangdong Basic and Applied Basic Research Foundation 2020A1515011417. Yongjie Wang  also thanks the support of National Natural Science Foundation of China (No. 11901146 and 12071026).


\begin{thebibliography}{15}
\bibitem{ChangWang2016} Z. Chang and Y. Wang, Central extensions of generalized ortho-symplectic Lie superalgebras, \textit{Sci. China Math.} {\bf 60} (2016), 223-260.
\bibitem{ChangChengWang2018}Z. Chang, J. Cheng and Y. Wang, Second homology of generalized periplectic Lie superalgebras. \textit{Linear Algebra Appl.} {\bf 546} (2018), 122-153.
\bibitem{ChenGuay2013} H. Chen and N. Guay, Central extensions of matrix Lie superalgebras over $\mathbb{Z}/2\mathbb{Z}$--graded algebras, \textit{Algebr. Represent. Theory} {\bf16} (2013), 591-604.
\bibitem{ChenSun2015} H. Chen and J. Sun, Universal central extensions of $\mathfrak{sl}_{m|n}$ over $\mathbb{Z}/2\mathbb{Z}$--graded algebras, \textit{J. Pure Appl. Algebra} {\bf 219} (2015), 4278-4294.
\bibitem{Gao1993} Y. Gao, On the Steinberg Lie algebras $St_2(R)$, \textit{ Comm. Algebra} {\bf 21} (1993), 3691-3706.
\bibitem{Gao1996} Y. Gao, Steinberg unitary Lie algebras and skew-dihedral homology, \textit{J. Algebra} {\bf 179} (1996), 261-304.
\bibitem{GL2017} X. García-Martínez and M. Ladra, Universal central extensions of $sl(m,n,A)$ over associative superalgebras. \textit{Turkish J. Math.} 41 (2017), 1552-1569. 
\bibitem{IoharaKoga2001} K. Iohara and Y. Koga, Central extensions of Lie superalgebras, \textit{Comment. Math. Helv.} {\bf76} (2001), 110-154.
\bibitem{IoharaKoga2005} K. Iohara and Y. Koga, Second homology of Lie superalgebras, \textit{Math. Nachr.} {\bf278} (2005), 1041-1053.
\bibitem{Kassel1986} C. Kassel, A K\"{u}nneth Formula for the Cyclic Homology of $\mathbb{Z}/2-$graded Algebras, \textit{Math. Ann.} {\bf 275} (1986), 683-699.
\bibitem{KasselLoday1982} J. L. Kassel and C. Loday, Extensions centrales d'alg\`{e}bres de Lie, \textit{Ann. Inst. Fourier (Grenoble)} {\bf32} (1982), 119-142 . 
\bibitem{MartinezZelmanov2003} C. Mart\'{i}nez and E. I. Zelmanov, Lie superalgebras graded by $P(n)$ and $Q(n)$, \textit{Proc. Natl. Acad. Sci. USA.} {\bf 100} (2003), 8130-8137.
\bibitem{MikhalevPinchuk2005} A. V. Mikhalev and I. A. Pinchuk, The universal central extension $Q(n,A)$ of Lie superalgebras. Chebyshevskiĭ Sb. 6 (2005), 149-153.
\bibitem{NY2020} K. H. Neeb and  M. Yousofzadeh, Universal central extensions of current Lie superalgebras, \textit{J. Pure Appl. Algebra} {\bf 224} (2020), no. 4, 106205, 10 pp. 
\bibitem{Neher2003} E. Neher, An introduction to universal central extensions of Lie superalgebras,  \textit{in Groups, rings, Lie and Hopf algebras,}  141-166, Math. Appl., 555, Kluwer Acad. Publ., Dordrecht, (2003).
\end{thebibliography}
\end{document}